\documentclass[10pt,reqno]{amsart}

\usepackage[margin=1in]{geometry}
\usepackage{amsmath,amssymb,amsthm,mathtools}
\usepackage{microtype}
\usepackage[hidelinks]{hyperref}

\hypersetup{colorlinks=true,linkcolor=red,citecolor=blue,urlcolor=blue}

\newtheorem{theorem}{Theorem}[section]
\newtheorem{proposition}[theorem]{Proposition}
\newtheorem{lemma}[theorem]{Lemma}
\newtheorem{corollary}[theorem]{Corollary}
\newtheorem{definition}[theorem]{Definition}
\newtheorem{assumption}[theorem]{Assumption}
\theoremstyle{remark}
\newtheorem{remark}[theorem]{Remark}

\newcommand{\C}{\mathbb C}
\newcommand{\E}{\mathbb E}
\newcommand{\R}{\mathbb R}
\newcommand{\Q}{\mathbb Q}
\newcommand{\D}{\mathcal D}
\newcommand{\Prob}{\mathbb P}
\newcommand{\dd}{\,d}
\newcommand{\logm}{\log^-}
\newcommand{\logp}{\log^+}
\newcommand{\weak}{\Rightarrow}

\title{Convergence for Small-Order Derivatives of Random Polynomials with Independent Roots}
\author{Hongkai Zhu}
\address{Department of Mathematics, University of Colorado, Boulder, CO}
\email{hongkai.zhu@colorado.edu}

\begin{document}

\begin{abstract}
Let $X_1,X_2,\ldots$ be i.i.d. complex-valued random variables with arbitrary Borel probability law $\mu$, and set $P_n(z)=\prod_{j=1}^n(z-X_j)$. For every deterministic sequence $k_n=o(n)$, we prove that the empirical zero measure of the $k_n$-th derivative $P_n^{(k_n)}$ converges weakly to $\mu$ almost surely.
\end{abstract}
\maketitle

\section{Introduction}
Let $X_1,X_2,\ldots$ be i.i.d. complex random variables with common law $\mu$. Consider a random polynomial
\[P_n(z):=\prod_{j=1}^n (z-X_j).\] 
For $0\le k<n$, let
\[\mu_{n,k}:=\frac{1}{n-k}\sum_{z:\in\C: P^{(k)}_n(z)=0}\delta_z\] where zeros are counted with multiplicity.

Our main result is the following.
\begin{theorem}\label{thm:main}
    Let $\mu$ be an arbitrary Borel probability measure on $\C$. If $k_n=o(n)$. Then 
    \[\mu_{n,k_n}\weak \mu\qquad\text{almost surely}.\]
\end{theorem}

Pemantle and Rivin~\cite{PR13} initiated the probabilistic study of critical points in this model and conjectured that, for $k=1$, the empirical critical-point measure converges weakly in probability to $\mu$; they proved this under the assumption that $\mu$ has finite $1$-energy. Subramanian~\cite{Sub12} proved convergence in probability for arbitrary laws supported on the unit circle, with almost-sure convergence in the nonuniform case. Kabluchko~\cite{Kab15} subsequently established the conjecture for arbitrary $\mu$; Angst, Malicet, and Poly~\cite{AMP24} later upgraded the first-derivative result to almost sure convergence. 

For higher derivatives, say $k=k(n)$-th derivatives, Cheung, Ng, Tsai, and Yam~\cite{CNTY15} established almost sure convergence for all unit-circle laws and for fixed $k$. Byun, Lee, and Reddy~\cite{BLR22} proved convergence in probability for fixed $k$; Michelen and Vu~\cite{MV24a,MV24b} proved convergence in probability up to $k\le \log n/(5\log\log n)$ and almost sure convergence for fixed $k$. 

The recent work of Angst, Nguyen, and Poly~\cite{ANP26} proves almost sure convergence for all $k_n=o(n)$ when $\mu$ is discrete, and for $k_n=o(n/\log n)$ when $\mu$ is dimension-nondegenerate. Our theorem extends the full \(k_n=o(n)\) range to arbitrary Borel probability measures.

\subsection{Proof strategy and organization}
The proof is based on logarithmic potentials from planar potential theory. Fix a subsequential vague limit $\beta$ of the $n$-normalized zero measures of $P_n^{(k_n)}$. After normalizing the polynomials and passing to a further subsequence, we obtain limiting logarithmic potentials $\varphi$ and $\psi$ for $P_n$ and $P_n^{(k_n)}$, respectively. Their normalized distributional Laplacians are $\mu$ and $\beta$. A deterministic estimate for the positive part of the logarithm of the derivative ratio implies
\[\psi\leq\varphi\qquad\text{almost everywhere}.\]

Thus it remains to prove that necessarily \(\beta=\mu\). To this end, define the nonnegative potential gap
\[\Gamma:=\phi-\psi\geq 0\qquad\text{almost everywhere},\]
and distinguish two cases according to whether \(\mu\) is polar-carried.

If $\mu$ is polar-carried, a rigidity argument based on the removability of polar singularities directly forces $\beta=\mu$. If $\mu$ is not polar-carried, the capacity dichotomy provides a compactly supported positive-mass component of $\mu$ whose logarithmic singularities are uniformly controlled. Were $\Gamma$ nonzero, Cauchy's estimates would force a consecutive block of normalized derivative ratios to be exponentially small on a set of positive planar Lebesgue measure. The postselection estimate of Section~3 shows that the expected area of such a set decays exponentially. This contradiction gives $\Gamma=0$, and hence again $\beta=\mu$.

Therefore every subsequential limit of the derivative-zero measures is equal to $\mu$. Convergence of their total masses then upgrades vague convergence to weak convergence.

Section~\ref{sec:prelim} develops the deterministic and potential-theoretic tools. Section~\ref{sec:non-polar} treats the non-polar-carried case using postselection and the derivative-block argument. Section~\ref{sec:polar} proves the rigidity result for the polar-carried case, and Section~\ref{sec:main-thm} assembles the proof of the main theorem.

Much of the potential-theoretic tools used below are taken from Ransford \cite{Ran95}.
\section{Potential theoretic and polynomial preliminaries}\label{sec:prelim}
For $x>0$, write 
\[\logm x:=\max\{0,-\log x\},\qquad \logp x:=\max\{0,\log x\}.\]

We use the convention that a function
$u:\Omega\to[-\infty,\infty)$ on an open set $\Omega\subset\C$ is
\emph{subharmonic} if it is upper semicontinuous and satisfies the mean value
inequality
\[u(z)\le
    \frac{1}{\pi r^2}\int_{B(z,r)}u(w)\,d^2w\]
whenever $\overline{B(z,r)}\subset\Omega$. We allow $u$ to take the value $-\infty$;
the case $u\equiv-\infty$ will be excluded explicitly when necessary.

\subsection{A positive-part estimate for the derivative ratio}
For $z\notin\{X_1,\ldots, X_n\}$, we have 
\[R_{n,k}(z):=\frac{P_n^{(k)}(z)}{(n)_kP_n(z)}=\binom{n}{k}^{-1}e_k\left(\frac{1}{z-X_1},\ldots, \frac{1}{z-X_n}\right),\]
where $(n)_k=n(n-1)\cdots (n-k+1)$ and $e_k$ is the elementary symmetric polynomial. We use the convention that $R_{n,0}\equiv 1$.

The positive part of the logarithm is controlled deterministically with no assumption on the roots.
\begin{lemma}\label{lem:positive}
For every compact $K\subset\C$, there is a constant $C_K<\infty$ such that, for every choice of $X_1,\ldots,X_n\in\C$ and every $0\leq k<n$,
\[
 \frac{1}{n}\int_K\logp|R_{n,k}(z)|\,\dd^2z
 \leq C_K\frac{k}{n}.
\]
\end{lemma}
\begin{proof}
    $k=0$ is immediate. For $k\geq 1$, Maclaurin's inequality gives
    \[|R_{n,k}(z)|\le\left(\frac{1}{n}\sum_{j=1}^n\frac{1}{|z-X_j|}\right)^k\] for Lebesgue-a.e. $z$. Since $\logp u\le u$ for $u\ge 0$,
    \[\logp|R_{n,k}(z)|\le \frac{k}{n}\sum_{j=1}^n\frac{1}{|z-X_j|}.\]
    We claim that $\sup_{w\in \C}\int_K\frac{1}{|z-w|}\,d^2z<\infty.$ Indeed, assume $K\subset B(0,R)$ and $|w|\le 2R$, then $K-w\subset B(0,3R)$; if $|w|>2R$, then $|z-w|\ge |w|-R>R$ on $K$ by the triangle inequality.

    Therefore,
    \[\frac{1}{n}\int_K\log^+|R_{n,k}(z)|\,d^2z
    \le \frac{k}{n^2}\sum_{j=1}^n\int_K\frac{d^2z}{|z-X_j|}
    \le C_K\frac{k}{n}.\]
\end{proof}

\subsection{Polar-carried measures and the capacity dichotomy}
We follow the terminology and conventions of \cite{Ran95}.
\begin{definition}[Polar carriers]
For a compactly supported probability measure $\nu$, define
\[I(\nu):=\iint_{\C\times\C}\log\frac{1}{|z-w|}\,d\nu(z)\,d\nu(w).\]
For a nonempty compact set $K\subset\C$, let $\mathcal M_1(K)$ denote the set of Borel probability measures supported on $K$, and define the logarithmic capacity of $K$ by
\[\operatorname{cap}(K):=\exp\left\{-\inf_{\nu\in\mathcal M_1(K)}I(\nu)\right\}.\]
For a Borel set $E\subset\C$, define
\[\operatorname{cap}(E) :=
 \sup\left\{\operatorname{cap}(K):
 K\subset E,\ K\text{ compact}\right\}.\]
The set $E$ is \emph{logarithmically polar} if $\operatorname{cap}(E)=0$. A probability measure $\mu$ is \emph{polar-carried} if $\mu(E)=1$ for some logarithmically polar Borel set $E$.
\end{definition}

For a compactly supported probability measure $\lambda$, put
\[L_\lambda:=\sup_{w\in\C}\int_\C\logm|x-w|\,d\lambda(x).\]

\begin{proposition}\label{prop:capacity-equivalence}
    A Borel probability measure $\mu$ is not polar-carried if and only if there exist $a\in(0,1]$ and a compactly supported probability measure $\lambda$ such that $a\lambda\le \mu$ and $L_\lambda<\infty$.
\end{proposition}
To prove Proposition~\ref{prop:capacity-equivalence}, we need the following lemma.
\begin{lemma}\cite{Ran95}\label{lem:polar-facts}
    For Borel subsets of $\C$, capacity-zero polarity agrees with the usual subharmonic notion of polarity. Under this identification, the following statements hold. 
    \begin{enumerate}
        \item [(1)][Theorem 3.1.2] If $\eta$ is a compactly supported finite positive measure on $\C$, then 
        \[p_\eta(z):=\int_\C \log|z-w|\,d\eta(w)\]
        is subharmonic on $\C$ and is not identically $-\infty$. 
        \item [(2)][Theorem 3.5.1] If $u$ is subharmonic on a domain and $u\not\equiv -\infty$, then $\{u=-\infty\}$ is polar.
        \item [(3)][Corollary 3.2.4] Every Borel polar set has planar Lebesgue measure zero.
        \item [(4)][Corollary 3.2.5] A countable union of Borel polar sets is polar.
    \end{enumerate}
\end{lemma}

\begin{remark}
    Every law with countable support is polar-carried, whereas a law with a nonzero component absolutely continuous with respect to planar Lebesgue measure is not. Singular continuous laws can occur on either side of this dichotomy. More information about polar sets can be found in \cite[Sections 3.2 and 5.1]{Ran95}.
\end{remark}

\begin{proof}[Proof of Proposition~\ref{prop:capacity-equivalence}]

For a compactly supported finite measure $\eta$, put
\[U_\eta(x):=
\int_{\C}\log^+\frac{1}{|x-y|}\,d\eta(y),\qquad S_\eta:=\{x\in\C:U_\eta(x)=\infty\}.\]
By Lemma~\ref{lem:polar-facts}(1), $p_\eta$ is subharmonic and is not identically $-\infty$. Since $\eta$ has compact support, 
$p_\eta(x)=\int_{\C}\log^+|x-y|\,d\eta(y)-U_\eta(x)$ with the first term finite for every $x$. Hence $S_\eta=\{x:p_\eta(x)=-\infty\}$ which is polar by Lemma~\ref{lem:polar-facts}(2).

Suppose $\mu$ is not polar-carried. For some integer $R\ge 1$, the finite measure $\eta:=\mu|_{\overline{B}(0,R)}$ is not carried by a polar set. Otherwise a countable union of polar carriers for these restrictions would carry $\mu$, contrary to Lemma~\ref{lem:polar-facts}(4). Since $S_\eta$ is polar, $\eta(\C\setminus S_\eta)>0$. Thus the Borel set $\{x:U_\eta\le M\}$ has positive $\eta$-mass for some $M<\infty$. By inner regularity, choose a compact $L\subset \{x:U_\eta\le M\}$ with $\eta(L)>0$ and set $\lambda_0:=\eta|_L$.

For $w\in \C$, choose $x\in L$ minimizing $|w-x|$. Then $|x-y|\le 2|w-y|$ for every $y\in L$. So 
\[\sup_{w\in \C}\int_L\logp\frac{1}{|w-y|}\,d\lambda_0(y)\le \sup_{w\in \C}\int_L\left(\log 2+\logp\frac{1}{|x-y|}\right)\,d\lambda_0(y)\le \lambda_0(\C)\log 2+ M<\infty.\] With $a=\lambda_0(\C)$ and $\lambda=\lambda_0/\lambda_0(\C)$ we have $a\lambda\le \mu$ and $L_\lambda<\infty$.

Conversely, suppose $a\lambda\le\mu$ and $L_\lambda<\infty$. By Tonelli's theorem and the definition of $L_\lambda$,

\[\iint_{\C^2}\log^+\frac1{|x-y|}\,d\lambda(x)d\lambda(y)
=\int_{\C}\left(\int_{\C}\log^-|x-y|\,d\lambda(x)\right)d\lambda(y)\\
\le L_\lambda<\infty.\]
On the other hand, compact support gives the negative part bounded. Hence
\[I(\lambda)=\iint_{\C^2}\log\frac1{|x-y|}\,d\lambda(x)d\lambda(y)<\infty.\]

Suppose by contradiction that $\mu$ is carried by a polar Borel set $E$. Since $a\lambda\le\mu$, $a\lambda(\C\setminus E)\le\mu(\C\setminus E)=0$ and hence $\lambda(E)=1$. Then there is a compact set $K\subset E$ such that $\lambda(K)>0$. Put $\nu:=\frac{\lambda|_K}{\lambda(K)}$. Then $I(\nu)<\infty$. Since $\nu$ is a probability measure supported on $K$, the definition of logarithmic capacity gives
\[\operatorname{cap}(K)
=\exp\left\{-\inf_{\rho\in\mathcal M_1(K)}I(\rho)\right\}
\ge e^{-I(\nu)}>0.\]
But $K\subset E$ and $E$ is polar, so $\operatorname{cap}(K)\le\operatorname{cap}(E)=0$, a contradiction. Thus $\mu$ is not polar-carried.

\end{proof}

\subsection{Compactness of logarithmic potentials}
We first introduce some notions for normalization needed when a vanishing proportion of roots is allowed to escape arbitrarily far. Let 
\[\mathcal{P}_n:=\{p\in \C[z]:\deg p\le n\}.\]
For $p(z)=\sum_{j=0}^n a_jz^j$, $q(z)=\sum_{j=0}^n b_j z^j$, we define the inner product and the induced norm
\[\langle p,q\rangle_{\mathrm{c},n}=\sum_{j=0}^n a_j\overline{b_j},\qquad
\|p\|_{\mathrm{c},n}^2=\langle p,p\rangle_{\mathrm{c},n}=\sum_{j=0}^n |a_j|^2.\]
Thus $1,z,\ldots, z^n$ form an orthonormal basis of $\mathcal{P}_n$. 
Also we have the two estimates
\begin{equation}\label{eq:estimate}
    |p(z)|\le \sqrt{n+1}\max\{1,|z|\}^n\|p\|_{\mathrm{c},n},\qquad p\in\mathcal{P}_n,\qquad \frac{1}{2\pi}\int_0^{2\pi}|p(e^{it})|^2\,dt=\|p\|_{\mathrm{c},n}^2.
\end{equation} They are given by Cauchy-Schwarz and Parseval's identity respectively.

Let $\Delta=\partial_x^2+\partial_y^2$ be the usual planar Laplacian. We use the normalized Laplacian
\[\widetilde{\Delta}:=\frac{1}{2\pi}\Delta,\]
so that $\widetilde{\Delta}\log|z-a|=\delta_a$.

For a nonzero polynomial $p$, write $[Z(p)]$ for its finite zero-counting measure, including multiplicity. 

Recall that vague convergence on $\C$ means convergence against all continuous compactly supported complex functions. A vague limit of positive measures of mass at most one has mass at most one. If finite positive measures converge vaguely while their total masses converge to the mass of the limit, then the convergence is weak.

We will need the following proposition in the later proofs. 

\begin{proposition}\cite{Dem12}\label{prop:subharmonic-tools}
    Let $\Omega\subset \C$ be open.
    \begin{enumerate}
        \item[(1)][Theorem 4.20] If a real distribution $T$ on $\Omega$ has $\widetilde{\Delta}T\ge 0$, then $T$ has a unique locally integrable subharmonic representative.
        \item[(2)][Proposition 4.21] If $(u_j)$ is a sequence of subharmonic functions on $\Omega$ such that 
        \[\sup_j \int_K|u_j(z)|\, d^2z<\infty\qquad\text{for every compact set }K\subset \Omega,\]
        then $(u_j)$ is relatively compact in $L_{\rm loc}^1(\Omega)$, and every $L_{\rm loc}^1$ limit has a subharmonic representative.
    \end{enumerate}
\end{proposition}
\begin{remark}
    Since $\widetilde{\Delta}$ is a positive multiple of the standard Laplacian, the positivity conventions agree.
\end{remark}

\begin{lemma}\label{lem:polynomial-compactness}
    Let $p_n\in\mathcal P_n\setminus\{0\}$ and $\|p_n\|_{\mathrm c,n}=1$.  Then $v_n(z):=\frac1n\log|p_n(z)|$ is relatively compact in $L^1_{\mathrm{loc}}(\C)$.  Every limit is finite almost everywhere. If along the same subsequence, $\frac{1}{n}[Z(p_n)] \to \sigma$ vaguely on $\C$, then every $L^1_{\mathrm{loc}}$ limit $v$ satisfies 
    \[\widetilde{\Delta}v=\sigma.\]
\end{lemma}

\begin{proof}
The Cauchy-Schwarz coefficient estimate in \eqref{eq:estimate} gives
\[v_n(z)\le \frac{\log(n+1)}{2n}+\log^+|z|,\]
so $(v_n)$ is locally uniformly bounded above. Parseval's identity in \eqref{eq:estimate} gives a point $\zeta_n$ on the unit circle such that $|p_n(\zeta_n)|\ge 1$ and hence $v_n(\zeta_n)\ge 0$.

Fix a compact $K\subset\C$ and choose $R>1+\sup_{z\in K}|z|$. Then, for every $n$, $K\subset B(\zeta_n,R)\subset B(0,R+1)$. Let $C_R\ge 0$ be a common upper bound for $v_n$ on $B(0,R+1)$. The submean inequality gives
\[0\le \pi R^2v_n(\zeta_n)
      \le \int_{B(\zeta_n,R)}v_n\,d^2z.\]
Thus, on $B(\zeta_n,R)$, the integral of the negative part of $v_n$ is at most that of its positive part, and the latter is at most $\pi R^2C_R$. Consequently,
\[\int_K |v_n|\,d^2z\le 2\pi R^2C_R.\]
Hence $(v_n)$ is bounded in $L^1_{\rm loc}(\C)$. Proposition~\ref{prop:subharmonic-tools}(2) therefore gives relative compactness in $L^1_{\rm loc}(\C)$ and every limit has a subharmonic representative. In particular, being locally integrable, it is finite almost everywhere.

Finally, factorization of $p_n$, together with $\widetilde{\Delta}\log|z-a|=\delta_a$, gives $\widetilde{\Delta} v_n=\frac{1}{n}[Z(p_n)]$ in the sense of distributions. Therefore if $v_n\to v$ in $L^1_{\mathrm{loc}}(\C)$ and simultaneously $\frac{1}{n}[Z(p_n)]\to\sigma$
vaguely along the same subsequence, both sides converge distributionally, and hence $\widetilde{\Delta}v=\sigma$.
\end{proof}

For $0\le k<n$, define the linear operator on $\mathcal{P}_n$ by
\[\D_{n,k}p:=\frac{p^{(k)}}{(n)_k}.\] Note that relative to the orthonormal monomial basis of $\mathcal{P}_n$, 
\[\D_{n,k}z^j=\begin{cases}\frac{(j)_k}{(n)_k}z^{j-k}\quad &j\ge k,\\
 \qquad 0\quad &j<k.\end{cases}\]
Consequently,
\begin{equation}\label{eq:bin-estimate}\ker \D_{n,k}=\operatorname{span}\{1,z,\ldots,z^{k-1}\},
 \qquad
 \binom{n}{k}^{-1}=\frac{k!}{(n)_k}\leq\frac{(j)_k}{(n)_k}\le 1
 \quad(k\le j\le n),
\end{equation}
with the convention that this span is $\{0\}$ when $k=0$. The nonzero singular values are $(j)_k/(n)_k$. In particular,
\begin{equation}\label{eq:contraction}
\|\D_{n,k}p\|_{\mathrm{c},n}\le \|p\|_{\mathrm{c},n}\qquad\text{for }p\in \mathcal{P}_n\qquad\text{and}\qquad \|\D_{n,k}h\|_{\mathrm{c},n}\ge\binom{n}{k}^{-1}\|h\|_{\mathrm{c},n}\qquad\text{for }h\perp\ker \D_{n,k}.
\end{equation}
If $k=k_n=o(n)$, then using the inequality $\binom{n}{k}\le (en/k)^k$, we obtain
\[\frac{1}{n}\log \frac{(n)_k}{k!}=\frac{1}{n}\log\binom{n}{k}\le \frac{k}{n}\log\frac{en}{k}\to 0.\]

For a degree-$n$ polynomial $P_n$, write
\[\alpha_n:=\frac{1}{n}[Z(P_n)].\]

\begin{lemma}\label{lem:no-escape}
    Let $P_n$ have degree $n$, let $k_n=o(n)$, and suppose $\alpha_n=\frac{1}{n}\sum_{j=1}^n \delta_{x_j}\weak \mu$ where $x_1,\ldots, x_n$ are roots of $P_n$ and $\mu$ is a probability measure. Then
    \[d_n:=\frac{1}{n}\log\frac{\|\D_{n,k_n}P_n\|_{c,n}}{\|P_n\|_{c,n}}\]
    is bounded. More precisely,
    \[\inf_n d_n>-\infty,\qquad d_n\le0\quad\text{for every }n.\]
\end{lemma}

\begin{proof}
    \eqref{eq:contraction} gives $d_n\le 0$. Suppose along a subsequence $d_n\to -\infty$. Then $k_n\ge 1$ eventually for sufficiently large $n$, otherwise $k_n=0$ implies $d_n=0$. Normalize $\|P_n\|_{\mathrm{c},n}=1$ and decompose orthogonally
    \[P_n=P_n^0+P_n^\perp,\qquad P_n^0\in\ker\D_{n,k_n},\quad P_n^\perp\perp \ker\D_{n,k_n}.\]
    The second part in \eqref{eq:contraction} and the definition of $d_n$ imply 
    \begin{equation}\label{eq:H-small}
    \|P_n^{\perp}\|_{\mathrm{c},n}\le 
    \binom{n}{k_n}\|\D_{n,k_n}P_n^\perp\|_{\mathrm{c},n} =
    \binom{n}{k_n}\|\D_{n,k_n}P_n\|_{\mathrm{c},n},\qquad
    \frac{1}{n}\log\|P_n^{\perp}\|_{\mathrm{c},n}\to -\infty.
    \end{equation}
    So $\|P_n^\perp\|_{\mathrm{c},n}\to 0$ hence $\|P_n^0\|_{\mathrm{c},n}\to 1$ by the Pythagorean theorem. Put $\widetilde{P}_n^0:=P_n^0/\|P_n^0\|_{\mathrm{c},n}$. Since $\deg\widetilde{P}_n^0<k_n=o(n)$, $\frac{1}{n}[Z(\widetilde{P}_n^0)]\to 0$ vaguely. By Lemma~\ref{lem:polynomial-compactness}, pass to a further subsequence on which
    \[\ell_n:=\frac{1}{n}\log|\widetilde{P}_n^0|\to \ell,\qquad s_n:=\frac{1}{n}\log|P_n|\to s\quad\text{in }L_{\rm loc}^1(\C)\text{ and almost everywhere}.\]
    The limits are finite almost everywhere. At every point where $\ell_n(z)\to \ell(z)>-\infty$, \eqref{eq:estimate} and \eqref{eq:H-small} give
    \[\frac{|P_n^\perp(z)|}{|\widetilde{P}_n^0(z)|}\le\sqrt{n+1}\max\{1,|z|\}^n\|P_n^\perp\|_{\mathrm{c},n}\exp(-n\ell_n(z))\to 0.\] This can be seen by taking the logarithm and dividing by $n$. Because $P_n=\|P_n^0\|_{\mathrm{c},n}\widetilde{P}_n^0+P_n^\perp$ and $\|P_n^0\|_{\mathrm{c},n}\to 1$, it follows that 
    \[s_n(z)-\ell_n(z)=\frac{1}{n}\log\left|\frac{P_n(z)}{\widetilde{P}_n^0(z)}\right|=\frac{1}{n}\log\left|\|P_n^0\|_{\mathrm{c},n}+\frac{P_n^\perp(z)}{\widetilde{P}_n^0(z)}\right|\to \frac{1}{n}\log|1+0|=0.\]
    Thus, $s=\ell$ almost everywhere. Applying $\widetilde{\Delta}$ gives 
    \[\mu=\widetilde{\Delta}s=\widetilde{\Delta}\ell=0\] 
    which contradicts $\mu(\C)=1$. So $(d_n)$ is bounded below.
\end{proof}

\section{The non-polar-carried case: postselection}\label{sec:non-polar}
\subsection{A polynomial postselection estimate}
The main probabilistic difficulty in this section is that the polynomial to which we apply a small-value estimate is not fixed in advance. More precisely, in Lemma~\ref{lem:postselection}, the linear space $V_N$ is fixed before the sample points $T_1,\ldots,T_{m_N}$ are sampled, but after observing these points one is allowed to choose a polynomial $v\in V_N$. We refer to this choice of $v$ after observing the sample as \emph{postselection}. Thus we must control the single existential event
\[\left\{\exists v\in V_N:v(0)=1,\quad|v(T_i)|\leq e^{-cN}\text{ for every }1\leq i\leq m_N\right\},\]
rather than merely obtain a probability bound for each fixed $v\in V_N$.

For a fixed polynomial $v$, define its small-value set by
\[E_v:=\{t\in\C:|v(t)|\leq e^{-cN}\}.\]
The proof of Lemma~\ref{lem:postselection} first shows that $\nu(E_v)$ is small, uniformly over all polynomials $v\in V_N$ satisfying $v(0)=1$. To pass from this pointwise estimate to a bound that is uniform over the entire family $\{E_v\}$, we control the combinatorial complexity of this family through its shatter function.

For a set $S$ and a class $\mathcal F$ of subsets of $S$, define
\[\mathsf S_{\mathcal F}(m):=\sup_{x_1,\ldots,x_m\in S}\# \left\{F\cap\{x_1,\ldots,x_m\}:F\in\mathcal F\right\}.\]
Thus $\mathsf S_{\mathcal F}(m)$ is the largest number of distinct membership patterns that sets in $\mathcal F$ can produce on $m$ prescribed points.

We use two standard combinatorial estimates. Warren's theorem (Lemma~\ref{lem:warren}) bounds the number of membership patterns when membership is described by polynomial inequalities in finitely many real parameters. The Vapnik--Chervonenkis inequality (Lemma~\ref{lem:vc-uniform}) then converts a bound on the shatter function into a uniform concentration estimate for empirical measures. In our application, the real and imaginary parts of the coefficients parametrizing \(v\) are the real parameters, and each condition
\[t_i\in E_v\iff|v(t_i)|^2-e^{-2cN}\leq 0\]
is a polynomial inequality of degree two in those parameters.

Here and below,$\Prob^*$ denotes outer probability:
\[\Prob^*(E):=\inf\{\Prob(A):E\subseteq A,\ A\text{ measurable}\}.\]
It agrees with $\Prob$ on measurable events. Moreover, the probability of any measurable subset of $E$ is at most $\Prob^*(E)$.
\begin{lemma}\cite[Theorem~3]{War68}\label{lem:warren}
    Let $f_1,\ldots, f_M$ be real polynomials of degree at most $D$ in $q\ge 1$ real variables. If $M\ge q$, the number of strict sign patterns
    \[\left(\operatorname{sgn}f_1(x),\ldots,\operatorname{sgn}f_M(x)\right)\in\{-1,1\}^M \] realized by points $x\in \R^q$ at which no $f_i$ vanishes is at most $(4eDM/q)^q$. 
\end{lemma}

\begin{lemma}\cite[Theorem~2]{VC71}.\label{lem:vc-uniform}
Let $S$ and $\Theta$ be standard Borel spaces, let $\nu$ be a Borel probability measure on $S$, and let $B\subset\Theta\times S$ be Borel. For $\theta\in\Theta$, set
\[F_\theta:=\{x\in S:(\theta,x)\in B\},
    \qquad \mathcal F:=\{F_\theta:\theta\in\Theta\}.\]
Let $T_1,\ldots,T_m$ be i.i.d. with law $\nu$, and write $\nu_m:=\frac{1}{m}\sum_{i=1}^m\delta_{T_i}$.
Then, for $0<\eta<1$ and $m\ge 2/\eta^2$,
\[\Prob^*\!\left(\sup_{F\in\mathcal F}|\nu_m(F)-\nu(F)|>\eta\right)
 \le
 4\mathsf{S}_{\mathcal F}(2m)\exp\left(-\frac{m\eta^2}{8}\right).\]
\end{lemma}

In the proof of Lemma~\ref{lem:postselection}, the affine slice
\[
\{v\in V_N:v(0)=1\}
\]
will be parametrized by \(\mathbb C^d\) for some
\(d\leq \dim V_N-1\). Under this parametrization, the incidence set
\[
\{(\alpha,t)\in\mathbb C^d\times\mathbb C:
|v_\alpha(t)|\leq e^{-cN}\}
\]
is closed. Hence the resulting family of small-value sets satisfies
the Borel-parametrization hypothesis of Lemma~\ref{lem:vc-uniform}.
The existential postselection event is a projection of a Borel set,
and is therefore analytic and universally measurable. Consequently,
on the completed probability space, its outer probability agrees
with its ordinary probability.

We now formulate the resulting postselection estimate.

\begin{lemma}[Polynomial postselection]\label{lem:postselection}
Let $K\subset\C$ be compact, fix $L<\infty$ and $\theta>0$, and let
$r_N=o(N)$. For every $c>0$, there exists $\rho>0$ depending only on $c, K, L$ with the following property.

Let $\nu$ be any Borel probability measure supported on $K$ such that $L_\nu=\sup_{w\in\C} \int_\C\logm|t-w|\,d\nu(t)\le L$.
For each positive integer $N$, let $M_N$ and $m_N$ satisfy $M_N\le\rho N$ and $m_N\ge\theta N$, let $V_N\subset\C[t]_{\le M_N}:=\{v\in \C[t]:\deg v\le M_N\}$ be a complex linear space with $\dim V_N\le r_N$, and let
$T_1,\ldots,T_{m_N}$ be i.i.d. with law $\nu$. Then
\[\Prob^*\!\left( \exists v\in V_N:\ v(0)=1,\quad |v(T_i)|\le e^{-cN}
     \text{ for every }1\le i\le m_N\right)
 \le e^{-c_1m_N}\]
for all sufficiently large $N$ where $c_1>0$, and one may take $c_1=1/64$.

The choice of $\rho$ and the estimate are uniform over all measures $\nu$
and spaces $V_N$ satisfying the preceding assumptions. The large-$N$
threshold may depend on the fixed parameters and on the sequence $(r_N)$.
\end{lemma}

\begin{proof}
Suppose that $v\in V_N$ satisfies $v(0)=1$. Write $\ell:=\deg v\le M_N$ and let $w_1,\ldots,w_\ell$ be its zeros, counted with multiplicity. Since $v(0)=1$, none of the $w_j$ is zero and $v(t)=\prod_{j=1}^{\ell}\left(1-\frac{t}{w_j}\right)$.
Therefore
\begin{align*}
    \int_\C\logm|v(t)|\,d\nu(t) 
    &\le \int_\C\sum_{j=1}^\ell \logp\frac{|w_j|}{|t-w_j|}\,d\nu(t)
    \le \int_\C\sum_{j=1}^\ell (C_K+\logm|t-w_j|)\,d\nu(t)\\
    &\le \sum_{j=1}^{\ell} \left(C_K+\int_\C\logm|t-w_j|\,d\nu(t)\right) 
    \le (C_K+L)\ell
    \le (C_K+L)M_N
\end{align*}
where $C_K:=\log(1+\sup_{z\in K}|z|)$ and the second inequality is induced by the triangle inequality applied on $|w|\le |t-w|+|t|\le|t-w|+\sup_{z\in K}|z|$. For
\[E_v:=\{t\in\C:|v(t)|\le e^{-cN}\},\]
since $\log^-|v|\geq cN$ on $E_v$, Markov's inequality gives $\nu(E_v)\le \frac{1}{cN}\int \logm |v(t)|\, d\nu(t)\le (C_K+L)\rho/c$. So we can choose $\rho$ such that $\nu(E_v)\le 1/16$ for every $v\in V_N$ satisfying $v(0)=1$.

If the affine slice $\{v\in V_N:v(0)=1\}$ is empty, then the event in the statement is empty and there is nothing to prove. Otherwise, choose $v_0\in V_N$ with $v_0(0)=1$, and let $u_1,\ldots,u_d$ be a basis of $\{u\in V_N:u(0)=0\}$. Every $v\in V_N$ satisfying $v(0)=1$ then has a unique representation
\[v_\xi = v_0+\sum_{j=1}^d\xi_j u_j, \qquad
    \xi=(\xi_1,\ldots,\xi_d)\in\C^d.\]
Since evaluation at the origin is a nonzero linear functional on $V_N$, $d=\dim V_N-1\le r_N-1$. Define
\[\mathcal E_N := \{E_{v_\xi}:\xi\in\C^d\} = \{E_v:v\in V_N,\ v(0)=1\}.\]

We next estimate the shatter function of $\mathcal E_N$. Fix $t_1,\ldots,t_{2m_N}\in\C$ and define
\[q_i(\xi) := |v_\xi(t_i)|^2-e^{-2cN}, \qquad 1\le i\le2m_N.\]
After identifying $\C^d$ with $\R^{2d}$, each $q_i$ is a real polynomial
of degree at most two, and $t_i\in E_{v_\xi} \iff q_i(\xi)\le0$. For each fixed $\xi$, choose $s>0$ smaller than every positive number among $q_1(\xi),\ldots,q_{2m_N}(\xi)$; if none of these numbers is positive, choose any $s>0$. Then none of the numbers $q_i(\xi)-s$ vanishes and $q_i(\xi)\le0 \iff q_i(\xi)-s<0$.
Thus every trace of $\mathcal E_N$ on $\{t_1,\ldots,t_{2m_N}\}$ is encoded by a strict sign pattern of the $2m_N$ polynomials
\[(\xi,s)\mapsto q_i(\xi)-s,
    \qquad 1\le i\le2m_N,\]
which have degree at most two in the $2d+1$ real variables $(\operatorname{Re}\xi,\operatorname{Im}\xi,s)$.

Since $d\le r_N-1$, $r_N=o(N)$, and $m_N\ge\theta N$, we have $2m_N\ge2d+1$ for all sufficiently large $N$. Applying Lemma~\ref{lem:warren} to these $2m_N$ polynomials of degree at most two in $2d+1$ real variables gives
\begin{equation}\label{eq:S-estimate}
    S_{\mathcal E_N}(2m_N) \le \left(\frac{16e m_N}{2d+1}\right)^{2d+1}.
\end{equation}
This estimate also covers $d=0$, with $\C^0$ understood as a singleton.

The map $(\xi,t)\longmapsto |v_\xi(t)|$ is continuous on $\C^d\times\C$. Hence the incidence set
\[\{(\xi,t)\in\C^d\times\C:t\in E_{v_\xi}\}
    =\{(\xi,t):|v_\xi(t)|\le e^{-cN}\}\]
is closed. Therefore $\mathcal E_N$ satisfies the Borel-parametrization
hypothesis of Lemma~\ref{lem:vc-uniform}. Let
\[\nu_{m_N}:=\frac{1}{m_N}\sum_{i=1}^{m_N}\delta_{T_i}\]
be the empirical measure of $T_1,\ldots,T_{m_N}$. On the event in the
statement, there exists $E_v\in\mathcal E_N$ such that $\nu_{m_N}(E_v)=1$ and $\nu(E_v)\le\frac1{16}$. Consequently,
\[\left\{\exists v\in V_N:\ v(0)=1, |v(T_i)|\le e^{-cN}
     \text{ for every }1\le i\le m_N\right\}
 \subset
 \left\{\sup_{E\in\mathcal E_N}|\nu_{m_N}(E)-\nu(E)|>\frac{1}{2}\right\}.\]
Since $m_N\ge\theta N$, Lemma~\ref{lem:vc-uniform} applies with
$\eta=1/2$ for all sufficiently large $N$. Therefore
\[\Prob^*\!\left(\exists v\in V_N:\ v(0)=1, |v(T_i)|\le e^{-cN}
     \text{ for every }1\le i\le m_N\right)
 \le 4S_{\mathcal E_N}(2m_N)e^{-m_N/32}.\]

Finally, since $d\le r_N-1$, $r_N=o(N)$, and $m_N\ge\theta N$, we have $(2d+1)/m_N\to 0$. Because $x\log(16e/x)\to 0$ as $x\to 0$, by \eqref{eq:S-estimate},
\begin{align*}
    4S_{\mathcal E_N}(2m_N)e^{-m_N/32}
    =4\exp\left(\log S_{\mathcal E_N}(2m_N)-\frac{m_N}{32}\right) 
    &\le
    4\exp\left((2d+1)\log\left(\frac{16e}{(2d+1)/m_N}\right)-\frac{m_N}{32}\right)\\
    &=4\exp\left(-\frac{m_N}{32}+o(m_N)\right)\le e^{-m_N/64}
\end{align*}
for all sufficiently large $N$. This proves the assertion with
$c_1=1/64$.
\end{proof}

\subsection{Small derivative blocks have exponentially small area}

For the rest of this section, we will use the following assumption.
\begin{assumption}\label{assump}
    Let $X_1,\ldots, X_n$ be i.i.d. with law $\mu$. Let $k_n=o(n)$ be deterministic and assume there exists $a\in (0,1]$ and a compactly supported probability measure $\lambda$ satisfying $a\lambda\le \mu$ and $L_\lambda<\infty$.
\end{assumption} 

Fix a bounded disk $D$, $\varepsilon>0$, $0<\delta<1/2$, and $T>0$. Define the deterministic good-anchor set
\[G_T:=\left\{z\in D: \mu(B(z,e^{-T}))\le e^{-T},\quad \lambda(B(z,e^{-T}))\le\frac{1}{2}\right\}.\] Fubini gives
\[\int_D \mu(B(z,e^{-T}))\,d^2z=\int_\C|D\cap B(x,e^{-T})|\,d\mu(x)\le \pi e^{-2T},\]
and the identical estimate holds with $\lambda$ in place of $\mu$. Markov's inequality therefore gives
\begin{equation}\label{eq:bad-size-bound}
\begin{aligned}
|D\setminus G_T|
&\le \left|\{z\in D:\mu(B(z,e^{-T}))>e^{-T}\}\right|
+\left|\{z\in D:\lambda(B(z,e^{-T}))>\tfrac12\}\right|\\
&\le e^T\int_D\mu(B(z,e^{-T}))\,d^2z 
+ 2\int_D\lambda(B(z,e^{-T}))\,d^2z
\le \pi e^{-T}+2\pi e^{-2T}
\le 3\pi e^{-T}.
\end{aligned}
\end{equation}
Here and below, for a Lebesgue-measurable set $A\subset\C$, $|A|$ denotes its planar Lebesgue measure. 

We henceforth take $n$ sufficiently large that 
\[k_n+\lfloor \delta n\rfloor\le n.\] For $z$ that is not an atom of $\mu$, define the event
\[B_n(z):=\{P_n(z)\neq 0, \quad
|R_{n,r}(z)|\le e^{-\varepsilon n} \quad\text{for every }
k_n\le r\le k_n+\lfloor \delta n\rfloor\}.\]
We set $B_n(z)=\emptyset$ at atoms of $\mu$. This changes no planar-area integral because the set of atoms is countable; the additional exclusion $P_n(z)=0$ likewise removes only finitely many anchors on each sample path. The maps $z\mapsto \mu(B(z,e^{-T}))$ and $z\mapsto \lambda(B(z,e^{-T}))$ are Borel, and the map $(\omega,z)\mapsto \mathbf{1}_{G_T}(z)\mathbf{1}_{B_n(z)}(\omega)$ is jointly measurable. Apart from the fixed countable set of atoms, it is determined by finitely many continuous polynomial evaluations and inequalities. Thus the random areas below are measurable and Tonelli's theorem applies.

The proof of the next lemma reduces $B_n(z)$ to the postselection event in Lemma~\ref{lem:postselection}. Fix $z$ and retain $m=\lceil \theta n\rceil$ eligible roots $Y_1,\ldots, Y_m$ and condition on all mixture labels and all unretained roots. Conditional on this background, the retained roots remain i.i.d. We then cancel the factors associated with the unretained roots outside $B(z,e^{-T})$. The resulting polynomial contains only the retained-root factors and the factors associated with the unretained roots inside $B(z,e^{-T})$, and therefore vanishes at 
\[t=Y_i-z,\quad 1\le i\le m.\]
Multiplication by the inverse product, followed by truncation at the degree bound established below, is a linear map determined by the conditioned background. The image under this map of $\C[t]_{k_n-1}$ is therefore a background-determined linear space $V$ satisfying $\dim V\le k_n=o(n)$. The contribution from the terms of degree less than $k_n$ gives a polynomial $v\in V$ with $v(0)=1$. On $B_n(z)$, the estimates proved below show that the remaining contribution is exponentially small at the points $Y_i-z$. Since the full polynomial vanishes at these points, it follows that 
\[|v(Y_i-z)|\le e^{-cn}, \quad 1\le i\le m,\]
for some $c>0$. Although $V$ is fixed after conditioning on the background, $v$ may depend on the retained roots. This is precisely the postselection allowed by Lemma~\ref{lem:postselection}.

\begin{lemma}\label{lem:summable}
    Under Assumption~\ref{assump}, for each fixed bounded disk $D$, $\varepsilon>0$, and $0<\delta<1/2$, there exists $T_0$ such that for every fixed $T\ge T_0$, there is $c_T>0$ for which 
    \[\E|\{z\in G_T:B_n(z)\}|\le |D|e^{-c_T n}\]
    for all sufficiently large $n$. In particular, these expectations are summable in $n$.
\end{lemma}

\begin{proof}
    Fix $z\in G_T$. There is nothing to prove when $z$ is an atom of $\mu$, otherwise $P_n(z)\neq 0$ almost surely. Also if $k_n=0$, then $B_n(z)=\emptyset$ because $R_{n,0}=1$. We therefore assume $z$ is non-atom and $k_n\ge 1$. Since $z\in G_T$,
    \[\lambda(\C\setminus B(z,e^{-T}))\ge\frac{1}{2},
    \qquad \mu(B(z,e^{-T}))\le e^{-T}.\]

    Set 
    \[S_z:=\C\setminus B(z,e^{-T}),\qquad p_z:=a\lambda(S_z).\]
    Since $a\lambda\le \mu$, we have $a\lambda\mid_{S_z}\le \mu$. Moreover, because $z\in G_T$, 
    \[\lambda(S_z)=1-\lambda(B(z,e^{-T}))\ge \frac{1}{2},\]
    and hence $p_z\ge a/2$. Define the probability measure 
    \[\lambda_z^{\mathrm{out}}:=\frac{\lambda\mid_{S_z}}{\lambda(S_z)}.\]
    If $p_z<1$, define 
    \[\rho_z:=\frac{\mu-a\lambda\mid_{S_z}}{1-p_z}.\]
    If $p_z=1$, let $\rho_z$ be any probability measure. In either case,
    \[\mu=p_z\lambda_z^{\mathrm{out}}+(1-p_z)\rho_z.\]

    We may therefore generate the roots independently as follows. For each $j$, let $\eta_j$ be Bernoulli with $\Prob(\eta_j=1)=p_z$. Conditional on $\eta_j=1$, sample $X_j$ with law $\lambda_z^{\mathrm{out}}$; conditional on $\eta_j=0$, sample $X_j$ with law $\rho_z$. The resulting variables $X_1,\ldots, X_n$ are still i.i.d. with law $\mu$. We call $X_j$ eligible when $\eta_j=1$. Thus the eligibility labels are independent, each root is eligible with probability $p_z\ge a/2$. Condition on all the labels, the eligible roots are independent with common law $\lambda_z^{\mathrm{out}}$.
    
    Let $m=\lceil \theta n\rceil$ where $\theta>0$ will be chosen later. Since the number of eligible roots has distribution $\operatorname{Bin}(n,p_z)$ with $p_z\ge a/2$, if $\theta<a/4$, Chernoff bound gives 
    \[\Prob(\text{there are fewer than }m\text{ eligible roots})=\Prob(\#\{1\le j\le n:\eta_j=1\}<m)\le e^{-an/16}.\]
    Since also $\mu(B(z,e^{-T}))\le e^{-T}$, Chernoff bound again gives that except on an event of probability at most $e^{-e^{-T}n/3}$ there are at most $2e^{-T}n$ roots in $B(z,e^{-T})$, i.e.
    \[\Prob(\#\{j:X_j\in B(z,e^{-T})\}>2e^{-T}n)\le e^{-e^{-T}n/3}.\]
    Call the roots in this ball the \emph{near roots}.
    
    On the intersection of these two good events, retain the first $m$ eligible roots and condition on all mixture labels and all unretained roots. Since every retained root is eligible, it lies outside $B(z,e^{-T})$; therefore the near-root count is determined entirely by the unretained roots and is part of the conditioned background. Thus, conditionally on this background, the retained roots $Y_1,\ldots, Y_m$ remain i.i.d. with law $\lambda_z^{\mathrm{out}}$. Therefore the translates $Y_i-z$ are conditionally i.i.d. with the corresponding translated law $\lambda_z^{\mathrm{out}}$. Define
    \[L:=\max\left\{1,\sup_{x\in \operatorname{supp} \lambda,z\in \overline{D}}|x-z|\right\}.\]
    Their law is supported on $\overline{B(0,L)}$ and satisfies, uniformly in $z\in D$,
    \[\sup_{w\in\C}\frac{1}{\lambda(S_z)}\int_{\C\setminus B(z,e^{-T})}\log^-|x-z-w|\,d\lambda(x)
    \le 2\sup_{u\in\C}\int_\C\log^-|x-u|\,d\lambda(x) 
    \le 2L_\lambda.\]
    Indeed, translation only changes the point in the supremum, restriction can only decrease the integral.

    Put
    \[F(t):=\prod_{j=1}^n\left(1+\frac{t}{z-X_j}\right)=\sum_{r=0}^n \binom{n}{r}R_{n,r}(z)t^r.\]
    Therefore on $B_n(z)$, the coefficient of $t^r$ satisfies $|\binom{n}{r}R_{n,r}|\le \binom{n}{r}e^{-\varepsilon n}$ for $k_n\le r\le k_n+\lfloor \delta n\rfloor$.
    We now invert only the factors corresponding to unretained roots outside $B(z,e^{-T})$. Define 
    \[Q(t):=\prod_{\substack{j\text{ unretained}\\
                     X_j\notin B(z,e^{-T})}}
    \left(1+\frac{t}{z-X_j}\right)^{-1}=\sum_{\ell\ge 0}b_\ell t^\ell.\]
    Each inverted factor has constant term $1$, so its formal inverse is well-defined. The series also converges absolutely for $|t|<e^{-T}$, but below we use only its coefficients and finite truncations. Because $Q$ depends only on the unretained roots, it is fixed once the background is conditioned on.

    Cancellation in the formal product give
    \[Q(t)F(t)=
    \prod_{i=1}^{m}\left(1+\frac{t}{z-Y_i}\right)\prod_{\substack{j\ \mathrm{unretained}\\X_j\in B(z,e^{-T})}}\left(1+\frac{t}{z-X_j}\right).\]
    Thus $QF$ is a polynomial. On the good counting event,
    \[\deg(QF)\leq M:=m+\lceil 2e^{-T}n\rceil.\]
    In particular, evaluating this polynomial gives
    \[(QF)(Y_i-z)=0,\qquad 1\leq i\leq m.\]
    These evaluations do not require convergence of the power series for $Q$ at $Y_i-z$.
    
    Every factor inverted in $Q$ corresponds to a root outside $B(z,e^{-T})$ and therefore satisfies $|\frac{1}{z-X_j}| \le e^T$. Expanding the reciprocal factors as geometric series gives
    \[|b_\ell| \le \binom{n+\ell-1}{\ell}e^{T\ell}, \qquad \ell\ge0.\]
    
    Write $(\cdot)_{\le M}$ for truncation through degree $M$, and define the background-determined linear space
    \[V:=\left\{(Qh)_{\le M}:h\in\C[t]_{\le k_n-1}\right\}
    \subset \C[t]_{\le M},\qquad\dim V\le k_n=o(n).\]
    Define 
    \[v:=\left(Q\sum_{s=0}^{k_n-1}\binom{n}{s}R_{n,s}(z)t^s\right)_{\le M}.\]
    Then $v\in V$ and $v(0)=1$.
    Although $V$ is fixed by the conditioned background, the particular polynomial $v$ may depend on the retained roots. This is precisely the postselection allowed by Lemma~\ref{lem:postselection}.
    
    Since $QF$ has degree at most $M$ on the good event, the definition of $v$ gives
    \[(QF-v)(t)
    =\left(Q(t)\sum_{s=k_n}^{n}\binom nsR_{n,s}(z)t^s\right)_{\le M} 
    =\sum_{r=k_n}^{M}\left(\sum_{s=k_n}^{r}b_{r-s}\binom nsR_{n,s}(z)\right)t^r.\]
    Suppose that $M\le k_n+\lfloor\delta n\rfloor$, as will follow from the parameter choice below. On $B_n(z)$, for every $k_n\le r\le M$, which is nonempty since $M\ge m\ge\theta n$ and $k_n=o(n)$,
    \[\left|\sum_{s=k_n}^{r}b_{r-s}\binom nsR_{n,s}(z)\right|\le(M+1)e^{-\varepsilon n}\max_{\substack{k_n\le s\le M\\0\le\ell\le M}}\binom ns\binom{n+\ell-1}{\ell}e^{T\ell}.\]
    Since both $QF$ and $v$ have degree at most $M$ and $(QF)(Y_i-z)=0$, on the intersection of $B_n(z)$ with the good event,
    \[|v(Y_i-z)|\le(M+1)^2L^M e^{-\varepsilon n}\max_{\substack{k_n\le s\le M\\0\le\ell\le M}}\binom ns\binom{n+\ell-1}{\ell}e^{T\ell},
    \qquad 1\le i\le m.\]
    Set $\alpha:=\theta+3e^{-T}$. For all sufficiently large $n$, $M/n\le\alpha$. If $\alpha<1/2$, the elementary binomial estimate $\binom Nr\le\left(\frac{eN}{r}\right)^r$ gives, for $0\le s,\ell\le M$,
    \[\binom ns\binom{n+\ell-1}{\ell}\le\exp\left\{Cn\alpha\log\frac e\alpha\right\}\]
    for some absolute constant $C>0$. And since $L^Me^{T\ell}\le\exp\{n\alpha(T+\log L)\}$, on $B_n(z)$ and the good event,
    \[|v(Y_i-z)|\le\exp\left(-n\left(\varepsilon-C\alpha\log\frac{e}{\alpha}-(T+\log L)\alpha-o(1)\right)\right),
    \qquad 1\le i\le m.\]
    Here the $o(1)$ absorbs the factor $(M+1)^2$ for all $z$.
    
    Let $\rho_*>0$ be supplied by Lemma~\ref{lem:postselection} with $c=\varepsilon/2$, compact support $\overline{B(0,L)}$, and potential bound $2L_\lambda$. Since $e^{-T}\log(e/e^{-T})\to 0$ and $Te^{-T}\to 0$, we may choose $T_0$ sufficiently large that, for every $T\ge T_0$,
    \[3e^{-T}<\min\{\rho_*,\delta/2,1/2\}
    \qquad\text{and}\qquad
    C(3e^{-T})\log\frac{e}{3e^{-T}}+(T+\log L)3e^{-T}<\frac{\varepsilon}{4}.\]

    Fix $T\ge T_0$. Since $x\log(e/x)\to 0$ as $x\to 0$ and $\alpha=\theta+3e^{-T}\to3e^{-T}$ as $\theta\to 0$, we may choose $\theta>0$ sufficiently small that
    \[\theta<\frac a4,
    \qquad
    \alpha<\min\{\rho_*,\delta/2,1/2\},
    \qquad
    C\alpha\log\frac e\alpha+(T+\log L)\alpha<\frac{\varepsilon}{3}.\]
    For all sufficiently large $n$, these choices imply $M\le\rho_*n$, $M\le k_n+\lfloor\delta n\rfloor$. They also give
    \[|v(Y_i-z)|\le e^{-\varepsilon n/2}, \qquad 1\le i\le m\] on $B_n(z)$ and the good event.
    
    Now fix any background satisfying the two good conditions. Under this conditioning, $V$ is fixed and $Y_1-z,\ldots,Y_m-z$ are i.i.d. with the translated law described above. Moreover, $V\subset\C[t]_{\le M}$, $\dim V\le k_n=o(n)$, $M\le\rho_*n$, $m\ge\theta n$. On the measurable section of $B_n(z)$, the preceding construction supplies a polynomial $v\in V$ satisfying $v(0)=1$ and $|v(Y_i-z)|\le e^{-\varepsilon n/2}$ for every $1\le i\le m$. 

    Therefore this measurable section is contained in the existential event of Lemma~\ref{lem:postselection}. Apply that lemma under the conditional product law of $Y_1-z,\ldots,Y_m-z$. Under this law, the probability of the measurable section is at most the outer probability of the containing event, and hence at most $e^{-c_1m}$. The estimate and the large-$n$ threshold are uniform over the conditioned backgrounds and $z\in G^T$. Since the two good conditions are determined by the background, averaging these ordinary conditional probabilities and adding the two Chernoff exceptions gives
    \[\Prob(B_n(z))
    \le e^{-c_1m}+e^{-an/16}+e^{-e^{-T}n/3}
    \le e^{-c_Tn}\]
    for some $c_T>0$ and all sufficiently large $n$. Finally Tonelli's theorem gives
    \[\E\bigl|\{z\in G_T:B_n(z)\}\bigr|=\int_{G_T}\Prob(B_n(z))\,d^2z \le|D|e^{-c_Tn}\]
    which completes the proof.
\end{proof}

\subsection{A potential gap forces a derivative block}
We now prove the deterministic implication that a nonzero limiting gap between the normalized logarithmic potentials of $P_n$ and $P^{(k_n)}_n$ forces exponentially small normalized derivatives for every order from $k_n$ through $k_n+\lfloor \delta n\rfloor$. This Cauchy-estimate argument is the sole step that produces a whole consecutive block of derivative orders.

\begin{lemma}\label{lem:spatial-gap-block}
    Let $P_n$ have degree $n$ and $k_n=o(n)$. Put
    \[\varphi_n(z):=\frac{1}{n}\log\frac{|P_n(z)|}{\|P_n\|_{\mathrm{c},n}},\qquad \psi_n(z):=\frac{1}{n}\log\frac{|P^{(k_n)}_n(z)|}{(n)_{k_n}\|P_n\|_{\mathrm{c},n}}.\]
    Suppose along a subsequence $\varphi_n\to\varphi$ and $\psi_n\to\psi$ in $L_{\rm loc}^1(\C)$, where $\psi\le \varphi$ almost everywhere and $\psi\neq \varphi$. Then there exists a bounded disk $D$, constants $\varepsilon$, $\delta>0$, a measurable set $E\subset D$ of positive area, and sets $E_n\subset E$ with $|E_n|\to |E|$ such that for all sufficiently large $n$,
    \[
    \left|\frac{P_n^{(k_n+j)}(z)}
    {(n)_{k_n+j}P_n(z)}\right|\le e^{-\varepsilon n},
    \qquad z\in E_n,\quad 0\le j\le\lfloor\delta n\rfloor.
    \]
\end{lemma}
\begin{proof}
    Since $P_n$ and $\mathcal D_{n,k_n}P_n$ are nonzero polynomials, $\varphi_n$ and $\psi_n$ are subharmonic. Moreover, first part of \eqref{eq:estimate} and \eqref{eq:contraction} give
    \[\varphi_n(z),\ \psi_n(z)
    \le \frac{\log(n+1)}{2n}+\log^+|z|,\]
    so both sequences are locally uniformly bounded above. Since $\varphi_n\to\varphi$ and $\psi_n\to\psi$ in $L^1_{\rm loc}(\C)$, their limits have subharmonic representatives.

    Since $\psi\le \varphi$ a.e. and $\psi\ne \varphi$, $\varphi-\psi$ is positive on a set of positive area. So we choose a $\gamma>0$ such that the set 
    \[H_\gamma:=\{\varphi-\psi>8\gamma,\, \varphi>-\infty,\,\psi>-\infty\}\]
    has positive area. The countable covering $H_\gamma\subset \bigcup_{t\in \mathbb Q}\{\psi<t,\,\varphi>t+6\gamma\}$ shows that $\psi<t,\varphi>t+6\gamma$ hold on a bounded set of positive area. Indeed, for every $z\in H_\gamma$, the interval $(\psi(z),\varphi(z)-6\gamma)$ is nonempty and therefore contains some $t\in\mathbb Q$. Since the resulting union is countable and $|H_\gamma|>0$, one of its members has positive area. Intersecting that member with a sufficiently large disk produces a bounded set of positive area on which both inequalities hold. For every point of this set, upper semicontinuity of $\psi$ gives a positive radius on which $\psi\le t+\gamma$. Refining over the countable radii $r=1/j$ and then using inner regularity gives a compact positive-area set $E$ and a common $r>0$ such that $\varphi>t+6\gamma$ on $E$. After enlarging a disk $D$, all these balls lie in $D$.

    Let
    \[U:=\bigcup_{z\in E}B(z,2r),\qquad K:=\bigcup_{z\in E}\overline{B(z,r)}.\]
    Then $U$ is open, $K$ is compact, and $\psi\le t+\gamma$ on $U$. Moreover, $B(w,r/2)\subset U$ for every $w\in K$. Indeed, if $w\in K$, then $|w-\zeta|\le r$ for some $\zeta\in E$. Therefore, if $|\xi-w|<r/2$, then $|\xi-\zeta|\le|\xi-w|+|w-\zeta|<\frac r2+r<2$r, so $\xi\in U$. Since $E$ is compact, $\overline{U}$ is compact. Hence $\psi_n\to \psi$ in $L_{\rm loc}^1(\C)$ implies $\int_U |\psi_n-\psi|\,d^2z\to 0$. By the mean value inequality,
    \[\psi_n(z)\le\frac{4}{\pi r^2}\int_{B(z,r/2)}\psi_n\,d^2z
    \le \frac{4}{\pi r^2}\int_{B(z,r/2)}\psi\,d^2z+\frac{4}{\pi r^2}\int_{B(z,r/2)}|\psi_n-\psi|\,d^2z
    \le t+\gamma+\frac{4}{\pi r^2}\int_{U}|\psi_n-\psi|\,d^2z.\]
    Therefore $\psi_n(z)\le t+2\gamma$ for sufficiently large $n$.
    Since $\varphi_n\to\varphi$ in measure on $E$, the set
    \[E_n:=\{z\in E:\varphi_n(z)\ge t+5\gamma\}\setminus\{P_n=0\}\] satisfies $|E_n|\to |E|$. Indeed, since $\varphi>t+6\gamma$ on $E$, we have 
    \[E\setminus\{z\in E:\varphi_n(z)\ge t+5\gamma\}\subset \{z\in E:|\varphi_n(z)-\varphi(z)|>\gamma\}.\]
    The area of the set on the right tends to zero. Moreover $\{P_n=0\}$ is finite and hence has area zero. 

    For $z\in E_n$, we have $\overline{B(z,r)}\subset K$. Since $\psi_n\le t+2\gamma$ on $K$,
    \[\sup_{w\in\overline{B(z,r)}}|\mathcal D_{n,k_n}P_n(w)|
    \le e^{n(t+2\gamma)}\|P_n\|_{\mathrm{c},n}.\]
    The definition of $E_n$ gives
    \[|P_n(z)|\ge e^{n(t+5\gamma)}\|P_n\|_{\mathrm{c},n}.\]
    Applying Cauchy's derivative estimate to $\mathcal D_{n,k_n}P_n$ therefore gives
    \[\left|\frac{P_n^{(k_n+j)}(z)}{(n)_{k_n+j}P_n(z)}\right|
    = \frac{|(\mathcal D_{n,k_n}P_n)^{(j)}(z)|}{(n-k_n)_j|P_n(z)|}
    \le \frac{j!}{(n-k_n)_j}r^{-j}\frac{e^{n(t+2\gamma)}\|P_n\|_{\mathrm{c},n}}{e^{n(t+5\gamma)}\|P_n\|_{\mathrm{c},n}}
    = \frac{j!}{(n-k_n)_j}r^{-j}e^{-3\gamma n}
    \le r^{-j}e^{-3\gamma n}\]
    for $0\le j\le n-k_n$. Here the first equality uses $(n)_{k_n+j}=(n)_{k_n}(n-k_n)_j$, and the last inequality follows from \eqref{eq:bin-estimate}.
    
    Choose $0<\delta<1/2$ so small that $\delta\log^+(1/r)<\gamma$. Since $k_n=o(n)$, for all sufficiently large $n$ we have $k_n+\lfloor\delta n\rfloor\le n$. Moreover, for $0\le j\le\lfloor\delta n\rfloor$,
    \[r^{-j}\le\exp\left(n\delta\log^+(1/r)\right)\le e^{\gamma n}.\]
    Plugging it back into the Cauchy estimate and taking $\varepsilon=2\gamma$ complete the proof.
\end{proof}

\begin{corollary}\label{cor:no-affine-gap}
    Under Assumption~\ref{assump}, almost surely the following holds: whenever along a subsequence $\varphi_n$, $\psi_n$ of Lemma~\ref{lem:spatial-gap-block} converge to $\varphi$, $\psi$ in $L^1_{\rm loc}(\mathbb C)$, respectively, and $\psi\le \varphi$ almost everywhere, one must have $\psi=\varphi$ almost everywhere.
\end{corollary}
\begin{proof}
    Fix a disk with rational center and radius, rational numbers $\varepsilon>0$ and $0<\delta<1/2$, and an integer $T\ge T_0$ where $T_0=T_0(D,\varepsilon,\delta)$ is supplied by Lemma~\ref{lem:summable}. Put
    \[H_n(D,\varepsilon,\delta,T):=\{z\in G_T:B_n(z)\}.\] Lemma~\ref{lem:summable} gives
    \[\sum_{n}\E|H_n(D,\varepsilon, \delta, T)|<\infty.\]
    Then Tonelli's theorem implies $\sum_n |H_n(D,\varepsilon, \delta, T)|<\infty$ almost surely. In particular $|H_n(D,\varepsilon, \delta, T)|\to 0$ with probability one. Intersecting these probability-one events over all $D\in \mathcal{D}:=\{B(a,R):a\in\Q+i\Q,\,R\in \Q_{>0}\},\, \varepsilon\in \Q_{>0},\,\delta\in\Q\cap(0,1/2)$, and all integers $T\ge T_0(D,\varepsilon, \delta)$, we obtain a single probability-one event on which 
    \[|H_n(D,\varepsilon, \delta, T)|\to 0\]
    for every such tuple.
    
    On the other hand, suppose on this event that a nonzero gap occurs, i.e. along a subsequence $\varphi_n\to \varphi$, $\psi_n\to\psi$, $\psi\le \varphi$ and $\psi\ne \varphi$, by Lemma~\ref{lem:spatial-gap-block}, there are $D,\,\varepsilon\,,\delta\,,E\,,E_n$ such that $|E_n|\ge |E|/2$ for all sufficiently large $n$ and the derivative block in $B_n(z)$ holds on $E_n$ for all large $n$. Removing from $E$ and $E_n$ the fixed countable set of atoms of $\mu$ does not change their planar Lebesgue measure. Choose a rational disk $D'\in \mathcal{D}$ containing $D$, together with rationals $0<\varepsilon'<\varepsilon$ and $0<\delta'<\min\{\delta,1/2\}$. The same block event $B_n(z)$ holds for this weaker tuple. Then choose an integer $T\ge T_0(D',\varepsilon',\delta')$ sufficiently large that \eqref{eq:bad-size-bound} gives $|D'\setminus G_T|<|E|/4$. Then 
    \[|H_n(D',\varepsilon',\delta',T)|\ge |E_n\cap G_T|\ge |E_n|-|D'\setminus G_T|\ge |E|/4\]
    for large $n$, contradicting the almost-sure conclusion above. Thus no nonzero gap can occur.
\end{proof}

\section{The polar-carried case: rigidity}\label{sec:polar}
We use below the removable-singularity theorem for subharmonic functions:

\begin{proposition}\cite[Theorem~3.6.1]{Ran95}\label{prop:polar-removability}
    Let $D$ be an open subset of $\C$, let $K$ be a closed polar set, and let $v$ be a subharmonic function on $D\setminus K$. If $v$ is locally bounded above near $K$, i.e. each point of $K\cap D$ has a neighborhood $N$ such that $v$ is bounded above on $N\setminus K$, then it has a unique subharmonic extension to $D$.
\end{proposition}

\begin{lemma}\label{lem:polar-rigidity}
    Let $\mu$ be a probability measure on $\C$ carried by a polar Borel set, and let $\beta$ be a positive Radon measure with $\beta(\C)\le 1$. If $u\in L^1_{\rm loc}(\C)$ satisfies
    \[u\le 0\quad\text{ almost everywhere,}\qquad \widetilde{\Delta}u=\beta-\mu,\]
    then 
    \[\beta=\mu.\]
\end{lemma}

\begin{proof}
    Let $E$ be a polar Borel set with $\mu(E)=1$. Fix a compact set $K\subset E$. Choose a bounded disk $D$ containing $K$, and put $\rho:=\mu|_{D\setminus K}$. Choose $M_D>\operatorname{diam}(D)$ and define
    \[G_\rho(z):=\int_D\log\frac{|z-q|}{M_D}\, d\rho(q),\qquad z\in D.\]
    The logarithmic kernel is locally integrable. Indeed, for every compact $A\subset D$, with Tonelli's theorem we have
    \[\int_A|G_\rho(z)|\,d^2z
    \le \int_D\int_A\log\frac{M_D}{|z-q|}\,d^2z\,d\rho(q)
    \le \rho(D)\int_{B(0,M_D)} \log\frac{M_D}{|u|}\,d^2u
    <\infty.\]
    Thus $G_\rho\in L^1_{\mathrm{loc}}(D)$. Moreover, $G_\rho\le 0$ almost everywhere on $D$ and $\widetilde{\Delta}G_\rho=\rho$. Hence $w:=u|_D+G_\rho\le 0$ almost everywhere on $D$ and 
    \[\widetilde{\Delta}w=(\beta|_D-\mu|_D)+\rho=(\beta|_D-\mu|_D)+(\mu|_D-\mu|_K)=\beta|_D-\mu|_K.\]
    On $D\setminus K$ this distribution is positive. Thus $w$ has a subharmonic representative there by Proposition~\ref{prop:subharmonic-tools}(1). Let's still denote it by $w$. Since it agrees almost everywhere with the original function, $w\le0$ almost everywhere on $D\setminus K$. For any $z_0\in D\setminus K$, choose $r>0$ such that $\overline{B(z_0,r)}\subset D\setminus K$. The mean value inequality gives
    \[w(z_0)\le\frac1{\pi r^2}\int_{B(z_0,r)}w\,d^2z \le 0.\]
    Thus $w\le0$ everywhere on $D\setminus K$, and in particular it is locally bounded above near $K$. Proposition~\ref{prop:polar-removability} extends its subharmonicity across $K$.

    By Lemma~\ref{lem:polar-facts}(3), $K$ has planar area zero. Hence the subharmonic extension agrees almost everywhere with the original function $w$ and therefore represents the same element of $L^1_{\mathrm{loc}}(D)$. Consequently, in the sense of distributions on $D$, $\widetilde{\Delta}w=\beta|_D-\mu|_K$. Since the extension is subharmonic on $D$, its normalized distributional Laplacian $\widetilde{\Delta}w$ is positive. Therefore $\beta|_D\ge\mu|_K$. Since $\mu|_K$ is supported in $D$, we moreover have $\beta\ge\mu|_K$ as measures on $\C$. We have therefore proved this inequality for every compact set $K\subset E$.

    By inner regularity and $\mu(E)=1$, choose compact sets $C_m\subset E$ such that $\mu(C_m)>1-2^{-m}$, and set $K_m:=C_1\cup\cdots\cup C_m$. Each $K_m$ is compact and polar by Lemma~\ref{lem:polar-facts}(4), so the preceding argument gives $\beta\ge\mu|_{K_m}$ for every $m$. The sets $K_m$ are increasing and $\mu(K_m)\to1$. Hence, for every Borel set $A\subset\C$,
    \[\beta(A)\ge\mu(A\cap K_m)\longrightarrow\mu(A).\]
    Thus $\beta\ge\mu$. Since $1=\mu(\C)\le\beta(\C)\le1$, we have $\beta(\C)=\mu(\C)$. The positive measure $\beta-\mu$ therefore has total mass zero, so we must have $\beta=\mu$.
\end{proof}

\section{Proof of the main theorem}\label{sec:main-thm}
For $0\le k<n$, define the $n$-normalized derivative-zero measure
\[\widetilde\mu_{n,k}:=\frac{1}{n}[Z(P_n^{(k)})]=\frac{n-k}{n}\mu_{n,k}.\]
Thus $\widetilde\mu_{n,k}(\C)=\frac{n-k}{n}\le 1$.
\begin{proof}[Proof of Theorem~\ref{thm:main}]
    Let $\Omega_0$ be the probability-one event on which
    \[\alpha_n:=\frac{1}{n}[Z(P_n)]=\frac{1}{n}\sum_{j=1}^n\delta_{X_j} \weak\mu.\]
    If $\mu$ is not polar-carried, Proposition~\ref{prop:capacity-equivalence} supplies $a\in(0,1]$ and a compactly supported probability measure $\lambda$ such that
    \[a\lambda\le\mu,\qquad L_\lambda<\infty.\]
    Thus Assumption~\ref{assump} holds, and Corollary~\ref{cor:no-affine-gap} supplies another probability-one event. In this case, replace $\Omega_0$ by its intersection with that event. Fix an outcome $\omega\in\Omega_0$.
    
    We first carry out the part of the argument common to both capacity alternatives. Fix an arbitrary subsequence. Since $\tilde\mu_{n,k_n}$ is a positive measure of total mass $(n-k_n)/n\le 1$, vague sequential compactness gives a further subsequence, which we relabel, such that $\tilde\mu_{n,k_n}\to \beta$ vaguely on $\C$ for some positive Radon measure $\beta$ with $\beta(\C)\le 1$. 

    Consider directly the two normalized logarithmic potentials appearing in Lemma~\ref{lem:spatial-gap-block}
    \[\varphi_n(z):=\frac{1}{n}\log\frac{|P_n(z)|}{\|P_n\|_{\mathrm{c},n}},\qquad \psi_n(z):=\frac{1}{n}\log\frac{|P^{(k_n)}_n(z)|}{(n)_{k_n}\|P_n\|_{\mathrm{c},n}}.\]
    We first extract joint $L_{\rm loc}^1$-limits of these functions. The sequence $(\varphi_n)$ is relatively compact in $L_{\rm loc}^1(\C)$ by Lemma~\ref{lem:polynomial-compactness}. For $\psi_n$, write 
    \[\psi_n(z)=\frac{1}{n}\log\left|\frac{\mathcal{D}_{n,k_n}P_n(z)}{\|\mathcal{D}_{n,k_n}P_n\|_{\mathrm{c},n}}\right|+d_n,\qquad\text{where }d_n:=\frac{1}{n}\log\frac{\|\mathcal{D}_{n,k_n}P_n\|_{\mathrm{c},n}}{\|P_n\|_{\mathrm{c},n}}.\]
    The first term is again relatively compact in $L_{\rm loc}^1(\C)$ by Lemma~\ref{lem:polynomial-compactness}, while Lemma~\ref{lem:no-escape} says that $(d_n)$ is bounded. Passing to a further subsequence, we may therefore assume that $\varphi_n\to \varphi$ and $\psi_n\to \psi$ in $L_{\rm loc}^1(\C)$ for some $\varphi,\psi\in L_{\rm loc}^1(\C)$. The ratio identity gives almost everywhere \[\psi_n-\varphi_n=\frac{1}{n}\log|R_{n,k_n}|.\]
    Let $K\subset \C$ be compact. Since the positive-part map is $1$-Lipschitz, together with Lemma~\ref{lem:positive} yields
    \[\int_K(\psi-\varphi)^+\,d^2z
    =\lim_{n\to\infty}\int_K(\psi_n-\varphi_n)^+\,d^2z
    =\lim_{n\to\infty}\frac{1}{n}\int_K\logp|R_{n,k_n}(z)|\,d^2z
    \le \lim_{n\to\infty} C_K\frac{k_n}{n}=0.\]
    Since $K$ was arbitrary, it follows that $\psi\le \varphi$ almost everywhere on $\C$. 

    We next identify the normalized distributional Laplacians of the two limits. The normalization constants in the definitions of $\varphi_n$ and $\psi_n$ are independent of $z$, so $\widetilde{\Delta}\varphi_n=\alpha_n$ and $\widetilde{\Delta}\psi_n=\tilde\mu_{n,k_n}$. Because $\varphi_n\to \varphi$ and $\psi_n\to \psi$ in $L_{\rm loc}^1$, we may apply $\widetilde{\Delta}$ in the sense of distributions. Using the weak convergence $\alpha_n\weak \mu$ and the vague convergence $\tilde\mu_{n,k_n}\to \beta$, we obtain $\widetilde{\Delta}\varphi=\mu$ and $\widetilde{\Delta}\psi=\beta$. 
    Set
    \[u:=\psi-\varphi.\]
    Then $u\le 0$ almost everywhere, and $\widetilde{\Delta}u=\beta-\mu$.

    \medskip
    \noindent
    \emph{Polar-carried case.}
    If $\mu$ is polar-carried, Lemma~\ref{lem:polar-rigidity} applied to $u$ gives $\beta=\mu$.

    \medskip
    \noindent
    \emph{Non-polar-carried case.}
    If $\psi\neq \varphi$ as elements of $L_{\rm loc}^1(\C)$, then the chosen subsequence satisfies
    \[\varphi_n\to \varphi,\qquad \psi_n\to \psi, \qquad \psi\le \varphi,\qquad \psi\neq \varphi,\]
    contradicting Corollary~\ref{cor:no-affine-gap}. Therefore $\psi=\varphi$ almost everywhere. Applying $\widetilde{\Delta}$ gives $\beta=\widetilde{\Delta}\psi=\widetilde{\Delta}\varphi=\mu$. 
    
    Thus, in either case, the vague limit $\beta$ is equal to $\mu$. Since the original subsequence was arbitrary, every subsequence of $(\tilde\mu_{n,k_n})$ has a further subsequence converging vaguely to $\mu$. It follows from the subsequence criterion that
    \[\tilde\mu_{n,k_n}\to\mu \qquad\text{vaguely on }\C.\]
    Moreover, $\tilde\mu_{n,k_n}(\C)=\frac{n-k_n}{n}\to 1=\mu(\C)$. Hence vague convergence, together with convergence of the total masses, implies
    \[\tilde\mu_{n,k_n}\weak\mu.\]
    
    Finally, since $\mu_{n,k_n}=\frac{n}{n-k_n}\tilde\mu_{n,k_n}$ and $\frac{n}{n-k_n}\to 1$, for every $f\in C_b(\C)$,
    \[\int_\C f\,d\mu_{n,k_n}=
    \frac{n}{n-k_n}\int_\C f\,d\tilde\mu_{n,k_n}\to
    \int_\C f\,d\mu.\]
    Consequently,
    \[\mu_{n,k_n}\weak\mu.\]
    
    Since the outcome fixed at the beginning of the proof was arbitrary in a probability-one event, this convergence holds almost surely.

\end{proof}

\section*{Acknowledgments}

This work was partially supported by the Research Experience for Undergraduates and Graduates program in the Department of Mathematics at the University of Colorado Boulder. The author thanks Sean O'Rourke for many helpful comments and suggestions on earlier versions of the manuscript, and Eskil Irgens for providing useful feedback.

\bibliographystyle{amsalpha}
\bibliography{sublinear_derivative_zeros}

\end{document}